\documentclass[reqno, a4paper, 11pt]{amsart}
\usepackage{amssymb}
\usepackage{mathrsfs}
\usepackage{amssymb, url, color, pb-diagram, graphicx, amscd, mathrsfs}
\usepackage[colorlinks=true, bookmarks=true, pdfstartview=FitH, pagebackref=true, linktocpage=true, linkcolor = magenta, citecolor = blue]{hyperref}
\usepackage[nodayofweek]{datetime}
\usepackage{graphicx}
\usepackage{times, txfonts}
\usepackage[pagewise,running,mathlines,displaymath]{lineno}
\usepackage[subnum]{cases}
\usepackage[compress]{cite}
\DeclareMathAlphabet{\mathpzc}{OT1}{pzc}{m}{it}

\newtheorem{theorem}{Theorem}[section]
\newtheorem{lemma}[theorem]{Lemma}
\newtheorem{corollary}[theorem]{Corollary}
\newtheorem{proposition}[theorem]{Proposition}

\theoremstyle{definition}

\theoremstyle{remark}
\newtheorem{remark}[theorem]{Remark}

\numberwithin{equation}{section}

\def\Xint#1{\mathchoice
 {\XXint\displaystyle\textstyle{#1}}%
 {\XXint\textstyle\scriptstyle{#1}}%
 {\XXint\scriptstyle\scriptscriptstyle{#1}}%
 {\XXint\scriptscriptstyle\scriptscriptstyle{#1}}%
 \!\int}
\def\XXint#1#2#3{{\setbox0=\hbox{$#1{#2#3}{\int}$}
 \vcenter{\hbox{$#2#3$}}\kern-.5\wd0}}

\def\dashint{\Xint-}

\newcommand{\definedas}{\mathrel{\raise.095ex\hbox{\rm :}\mkern-5.2mu=}}

\let\epsilon\varepsilon

\allowdisplaybreaks

\begin{document}

\title[Scalar Curvature Flow]{Scalar Curvature Flow on $S^n$ to a Prescribed Sign-changing Function}
% Remove any unused author tags.

% author one information
%\author[Q.A. Ng\^{o}]{Qu\^{o}\hspace{-0.5ex}\llap{\raise 1ex\hbox{\'{}}}\hspace{0.5ex}c Anh Ng\^{o}}
%\address[Q.A. Ng\^{o}]{Laboratoire de Math\'ematiques et de Physique Th\'eorique \\ Universit\'e Fran\c cois Rabelais de Tours \\ Parc de Grandmont\\ 37200 Tours \\ FRANCE}
%\email{\href{mailto: Q. A. Ngo <Quoc-Anh.Ngo@lmpt.univ-tours.fr>}{quoc-anh.ngo@lmpt.univ-tours.fr}}
%\email{\href{mailto: Q. A. Ngo <bookworm\_vn@yahoo.com>}{bookworm\_vn@yahoo.com}}

 %author two information
\author[H. Zhang]{Hong Zhang}
\address[H. Zhang]{School of Mathematics, University of Science and Technology of China, No.96 Jinzhai Road, Hefei, Anhui, China, 230026.}
%\email{\href{mailto: H. Zhang <hzhang@nus.edu.sg>}{hzhang@nus.edu.sg}}
\email{\href{mailto: H. Zhang <matzhang@ustc.edu.cn>}{matzhang@ustc.edu.cn}}
\thanks{}



\date{\today \ at \currenttime}

%\dedicatory{}

\setpagewiselinenumbers
\setlength\linenumbersep{110pt}
%\linenumbers

\maketitle
\begin{abstract}
  In this paper, we consider the problem of prescribing scalar curvature on n-sphere. Assume that the candidate curvature function $f$, which is allowed to change sign, satisfies some kind of Morse index or symmetry condition. By studying the well-known scalar curvature flow, we are able to prove that the flow converges to a metric with the prescribed function $f$ as its scalar curvature.
\end{abstract}

\section{Introduction}
The problem of prescribing certain curvature on a compact manifold with or without boundary has always been one of the most active topics in conformal geometry during the past few decades, see for instance \cite{bfr, br, cl, cgy, cy1, cy2, mu, xz} and references therein. Among them, a typical model is the prescribed scalar curvature problem on a close manifold of dimension $n\geqslant3$, which can be described as follows

Let $(M,g_0)$ be an $n\geqslant3$ dimensional and close manifold with the background Riemannian metric $g_0$ and $f$ a smooth function on $M$. Then, one may ask whether $f$ can be realized as the scalar curvature of some metric $g$ conformally related to the metric $g_0$. If we write $g=u^{4/(n-2)}g_0$, then the problem of prescribing scalar curvature is equivalent to finding a positive solution of the semi-linear PDE:
\begin{equation}
  \label{sce0}
  -c_n\Delta_{g_0}u+R_0u=f(x)u^{2^*-1},
\end{equation}
where $c_n=4(n-1)/(n-2)$, $2^*=2n/(n-2)$, $\Delta_{g_0}$ and $R_0$ are, respectively, the Laplace-Beltrami operator and the scalar curvature of the background metric $g_0$.

It is well known that the nature of the solutions of Eq.\eqref{sce0} depends on the so-called Yamabe invariant which is given by
\begin{equation*}
  Y(M)=\underset{u\in C^\infty(M),u>0}{\inf}\frac{\int_{S^n}c_n|\nabla u|_{g_0}^2+R_0u^2~d\mu_{g_0}}{\bigg(\int_{S^n}u^{2^*}~d\mu_{g_0}\bigg)^\frac{n-2}{n}}
\end{equation*}
Loosely speaking, the case of $Y(M)<0$ is well understood by a series of works due to Kazdan-Warner \cite{K-W-1975}, Ouyang \cite{O-1991,O-1992} and Rauzy \cite{ra}. When $Y(M)=0$, Kazdan-Warner \cite{K-W-1975} conjectured that if $f\not\equiv0$, then Eq.\eqref{sce0} possesses a positive solution if and only if $\max f>0$ and $\int_Mf~d\mu_{g_0}<0$. Schoen \& Escobar \cite{es} confirmed this conjecture for $n=3$ or $4$. Later, Jung \cite{jung} claimed to solve the conjecture by the method of sup-sub solutions. However, Druet \cite{druet} found a serious gap in his proof. As far as we know, this conjecture has been open until now for $n\geqslant5$. For the case of $0<Y(M)<Y(S^n)$, Schoen \& Escobar \cite{es} showed that if $n=3$, then any prescribed function $f$ positive somewhere can be the scalar curvature of some conformal metric on $M$. While for higher dimensional case, they require $f$ to satisfy some kind of flatness condition: all derivatives up to $n-2$ order vanish at some maximum point of $f$. The most subtle case is when the underlying manifold is the standard $n$-sphere. In this case the equation \eqref{sce0} becomes
\begin{equation}
  \label{sce}
-c_n\Delta_{S^n}u+n(n-1)u=f(x)u^{2^*-1},
\end{equation}
This equation has been extensively studied and various results have been known. Among many others, we refer the reader to \cite{bc, bo, cy3, le, li1, li2, ma, sz} and the literature therein. One of interesting researches among them is due to Chen \& Xu \cite{cx}. To state their result, we firstly define the map from the $n+1$-dimensional ball $\mathbb{B}^{n+1}$ to $\mathbb{R}^{n+1}$ by
$$G(p,s)=\dashint_{S^n}xf\circ\phi_{p,s}~d\mu_{S^n}.$$
where $(p,s)\in S^n\times[1,+\infty)$ identified with $\mathbb{B}^{n+1}$ through the map $(p,s)\in S^n\times[1,+\infty)\mapsto (s-1)p/s$ and $\phi_{p,s}$ is the conformal transformation on $S^n$. This definition of the map $G$ is given, by Chang \& Yang, in \cite{cy3}. Now, it is the right time to describe their result. Specifically, they have the following result
\begin{theorem}[Chen and Xu]
  \label{cx}
  Suppose $n\geqslant3$ and $f: S^n\rightarrow \mathbb{R}$ is a smooth positive Morse function with only non-degenerated critical points. Further, suppose $\delta_n=2^{2/n}$ if $n=3, 4$ or $=2^{2/(n-2)}$ if $n\geqslant5$. If $f$ satisfies the simple bubble condition, namely,
  \begin{equation}\label{sbc}
    \max_{S^n}f/\min_{S^n}f<\delta_n,
  \end{equation}
   and the degree condition
   \begin{equation}\label{degree}
   \mbox{Deg}(G,\mathbb{B}^{n+1},0)\neq0,
   \end{equation}
   then there exists a positive solution of the scalar curvature equation \eqref{sce}.
\end{theorem}
To prove this theorem, they studied carefully the scalar curvature flow
\begin{equation}\label{flow}
  \frac{\partial}{\partial t}g(t)=-\big(R_{g(t)}-\alpha(t)f\big)g(t),
\end{equation}
where $\alpha(t)$ is chosen to fix the volume of $g(t)$ along the flow. We should point out that using the method of flow to study prescribed curvature problems originated in \cite{br} by S. Brendle. By a contradiction argument plus an infinitely dimensional Morse theory trick due to Malchiodi \& Struwe \cite{ms}, they succeeded in showing the convergence of the flow \eqref{flow}. One thing we have to clarify here is that, instead of condition \eqref{degree}, they used the Morse index condition \eqref{Morseindex} below in proving their result. They, in fact, showed that conditions \eqref{degree} and \eqref{Morseindex} are equivalent under the assumptions on $f$ in the theorem (see \cite[Remark 1.2]{cx}).

When the prescribed function $f$ possesses some kind of symmetry, for instance, reflection or rotation, Leung \& Zhou \cite{lz} obtained an interesting result. Their proof is heavily relying on the asymptotic behavior of the flow \eqref{flow} proved by Chen \& Xu in \cite{cx}. Let us state Leung \& Zhou's result
\begin{theorem}[Leung and Zhou]
  \label{lz}
  Suppose $f>0$ is a smooth function on $S^n$ which is invariant under a mirror symmetry upon a hyperplane $\mathscr{H}\subset\mathbb{R}^{n+1}$ ($\mathscr{H}$ passes through the origin) or a rotation of angle $\theta/k$ (with axis being a straight line in $\mathbb{R}^{n+1}$ passing through the origin and $k>1$ being an integer). Let $\Sigma$ be the fixed points set under the action of symmetries above. Assume that there exists a point $y\in\Sigma$ with $f(y)=\max_\Sigma f$ and $\Delta_{S^n}f(y)>0$ and that
  \begin{equation}\label{sbc1}
  \max_{S^n}f/\max_\Sigma f<2^{2/(n-2)},
  \end{equation}
  then there exists a smooth positive solution of Eq.\eqref{sce}.
\end{theorem}

The aim of the current paper is, by reconsidering the scalar curvature flow \eqref{flow}, to generalize the results in Theorems \ref{cx} and \ref{lz} to the case that the prescribed function $f$ is allowed to change sign. Precisely speaking, our first result can be stated as follow
\begin{theorem}\label{main}
 Suppose $f(x)$ is a smooth Morse function on $S^n$ satisfying the following conditions:
 \begin{itemize}
   \item[(i)]
   $
   \begin{cases}
    f(x)\geqslant0~~\mbox{on}~~S^n~~\mbox{and}~~\int_{S^n}f~d\mu_{S^n}>0,&\mbox{if}~~n=3,\\[0.3em]
    \int_{S^n}f~d\mu_{S^n}>0,&\mbox{if}~~n\geqslant4;
   \end{cases}
   $
   \item[(ii)]
    $(\max_{S^n}|f|)\big/\big(\dashint_{S^n}f~d\mu_{S^n}\big)<2^\frac{2}{n}$;
   \item[(iii)] $|\nabla f|_{g_{S^n}}^2+(\Delta_{S^n}f)^2\neq0$;
   \item[(iv)] The following algebraic system has no non-trivial solutions
   \begin{equation}\label{Morseindex}
   m_0=1+k_0, m_i=k_{i-1}+k_i, 1\leqslant i\leqslant n, k_n=0,
   \end{equation}
   with coefficients $k_i\geqslant0$ and $m_i$ defined as
   \begin{equation}\label{Morseindex1}
   m_i=\#\{x\in S^n; f(x)>0, \nabla_{g_{S^n}}f(x)=0, \Delta_{g_{S^n}}f(x)<0,~~\mbox{ind}_f(x)=n-i\},
   \end{equation}
   where $\mbox{ind}_f(x)$ denotes the Morse index of $f$ at critical point $x$,
 \end{itemize}
then $f$ can be realized as the scalar curvature of some metric $g$ in the conformal class of $g_{S^n}$, i.e., Eq.\eqref{sce} possesses a positive solution.
\end{theorem}
As explained by Machiodi \& Struwe in \cite{ms}, the index counting condition \eqref{indexcounting} below is, indeed, a special case of the Morse index condition \eqref{Morseindex}. Hence, we have the following corollary 
\begin{corollary}\label{corollary}
  Suppose $f(x)$ is a smooth Morse function on $S^n$ satisfying the following conditions:
 \begin{itemize}
   \item[(i)]
   $
   \begin{cases}
    f(x)\geqslant0~~\mbox{on}~~S^n~~\mbox{and}~~\int_{S^n}f~d\mu_{S^n}>0,&\mbox{if}~~n=3,\\[0.3em]
    \int_{S^n}f~d\mu_{S^n}>0,&\mbox{if}~~n\geqslant4;
   \end{cases}
   $
   \item[(ii)]
    $(\max_{S^n}|f|)\big/\big(\dashint_{S^n}f~d\mu_{S^n}\big)<2^\frac{2}{n}$;
   \item[(iii)] $|\nabla f|_{g_{S^n}}^2+(\Delta_{S^n}f)^2\neq0$;
   \item[(iv)] 
   \begin{equation}\label{indexcounting}
   \sum_{\{x\in S^n: f(x)>0, \nabla_{S^n}f(x)=0~~\mbox{and}~~\Delta_{S^n}f(x)<0\}}(-1)^{ind_f(x)}\neq(-1)^n,
   \end{equation}
 \end{itemize}
then $f$ can be realized as the scalar curvature of some metric $g$ in the conformal class of $g_{S^n}$, i.e., Eq.\eqref{sce} possesses a positive solution.
\end{corollary}

For the case that the prescribed function $f$ possesses some kind of symmetry, inspired by the work \cite{bfr}, it seems that we could consider a little more general situation than that in Theorem \ref{lz}. Now, let us describe the idea by setting up some notations first. Let $G$ be a subgroup of isometry group of $S^n$. We say a function $f$ is $G-$invariant if
$$f(\theta(x))=f(x),\quad\mbox{for}~~\forall\theta\in G~~\mbox{and}~~\forall x\in S^n.$$
In addition, we define $\Sigma$ to be the fixed point set under the group $G$ as follow
$$\Sigma=\{x\in S^n: \theta(x)=x,\quad\mbox{for all}~~\theta\in G.\}.$$
Now, our second result reads as
\begin{theorem}\label{main1}
  Let $G$ be a subgroup of isometry group of $S^n$. Assume that $f$ is a $G-$invariant function satisfying
  \begin{itemize}
   \item[(i)]
   $
   \begin{cases}
    f(x)\geqslant0~~\mbox{on}~~S^n~~\mbox{and}~~\int_{S^n}f~d\mu_{S^n}>0,&\mbox{if}~~n=3,\\[0.3em]
    \int_{S^n}f~d\mu_{S^n}>0,&\mbox{if}~~n\geqslant4;
   \end{cases}
   $
   \item[(ii)]
    $(\max_{S^n}|f|)\big/\big(\dashint_{S^n}f~d\mu_{S^n}\big)<2^\frac{2}{n}$;
 \end{itemize}
 If there holds either
 \begin{itemize}
 \item[(a)] $\Sigma=\emptyset$, or\\ \vspace{-0.7em}
 \item[(b)] $\Sigma\neq\emptyset$ and one of two alternatives holds: $1^\circ$. $\max_{\Sigma}f\leqslant\dashint_{S^n}f~d\mu_{S^n}$; $2^\circ$. there exists a point $x_*\in \Sigma$ with $f(x_*)=\max_\Sigma f$ such that $\Delta_{S^n}f(x_*)>0$,
 \end{itemize}
 then $f$ can be realized as the scalar curvature of some metric $g$ in the conformal class of $g_{S^n}$, i.e., Eq.\eqref{sce} possesses a positive $G$-invariant solution.
\end{theorem}
\vspace{0.3em}

\begin{remark}~~
\begin{itemize}
  \item[1)] If $f>0$ on $S^n$, the condition (ii) in Theorem \ref{main} becomes ${\max_{S^n}f}\big/({\dashint_{S^n}f~d\mu_{S^n}})$\\$<2^{2/n}$. Comparing to the assumption \eqref{sbc} in Theorem \ref{cx}, ours is weaker in the sense that even though ${\max_{S^n}f}\big/({\dashint_{S^n}f~d\mu_{S^n}})<2^{2/n}$, the quotient $\max_{S^n}f/(\min_{S^n}f)$ may be large.

  \item[2)] In view of the assumption that $\max_{S^n}f/\max_{\Sigma}f<2^{2/(n-2)}$ in Theorem \ref{lz}, our assumption (ii) in Theorem \ref{main1} does not need to involve any information of the fixed points set $\Sigma$. The reason is as follows. By combining the conditions $1^\circ$ and $2^\circ$, it is not hard to see that the assumption $\Delta_{S^n}f(x_*)>0$ will only be used in the case of $\max_{\Sigma}f>\dashint_{S^n}f~d\mu_{S^n}$ (as if $\max_{\Sigma}f\leqslant\dashint_{S^n}f~d\mu_{S^n}$, then the conclusion of the theorem holds). However, when $\max_{\Sigma}f>\dashint_{S^n}f~d\mu_{S^n}$, we will have ${\max_{S^n}f}/\max_{\Sigma}f<{\max_{S^n}f}\big/({\dashint_{S^n}f~d\mu_{S^n}})$. So, once ${\max_{S^n}f}\big/({\dashint_{S^n}f~d\mu_{S^n}})<2^{2/n}$, the condition \eqref{sbc1} in Theorem \ref{lz} will be satisfied automatically.

  \item[3)] In the case of $f$ changing sign, we require the dimension of $S^n$ to be at least $4$. This is because when $n=3$, we cannot get the lower bound of the quantity $\lambda^\prime(t)$, which is crucial for the global existence of the flow \eqref{eeforu} below.

  \item[4)] X. Xu and the author \cite{xz} generalized the result in Theorem \ref{cx} to the case of prescribing mean curvature on the unit ball. It is natural to believe those results in Theorems \ref{main} and \ref{main1} should also hold in the case of prescribed mean curvature on the unit ball. This is our forthcoming paper \cite{hz}

  \end{itemize}
\end{remark}

The paper is organized as follows: In \S2, we derive some evolution equations and elementary estimates; In \S3, we mainly focus on the global existence of our flow; In \S4, we try to perform the so-called blow-up analysis and describe the asymptotic behavior of the flow in the case of divergence; In the final section \S5, we devote ourselves to proof of the main theorems in the paper.

\section{the flow equation and some elementary estimates}
\subsection{The flow equation}
As in Chen \& Xu \cite{cx}, we consider a family of time-dependent metrics conformal to $g_{S^n}$ whose evolution equation satisfies
\begin{equation}\label{eeforg}
\begin{cases}
  \frac{\partial}{\partial t}g(t)=-(R-\lambda(t)f)g(t),\\[0.3em]
  g(0)=g_0,
\end{cases}
\end{equation}
where $R$ is the scalar curvature of $g(t)$, $g_0$ is the inital metric in the confomal class of $g_{S^n}$ and $\lambda(t)$ is given by \eqref{formulaforlambda} below.

The metric flow \eqref{eeforg} preserves the conformal class. So, if we write $g(t)=u(t)^{4/(n-2)}g_{S^n}$ and $g_0=u_0^{4/(n-2)}g_{S^n}$, then we have the evolution equation for the conformal factor $u(t)$
\begin{equation}
  \label{eeforu}
\begin{cases}
  \frac{\partial}{\partial t}u(t)=-\frac{n-2}{4}(R-\lambda(t)f)u(t),\\[0.3em]
  u(0)=u_0,
\end{cases}
\end{equation}
where $R$ can be written, in terms of $u$, as
\begin{equation}
  \label{formulaforR}
  R=u^{2^*-1}(-c_n\Delta_{g_{S^n}}u+n(n-1)u).
\end{equation}

{\bf Convention}: From now on, the average sign `$\dashint_{S^n}$' means that
$$\dashint_{S^n}=\frac{1}{\omega_n}\int_{S^n},$$
where $\omega_n$ is the volume of $S^n$ w.r.t. the standard metric $g_{S^n}$.

Recall that \eqref{sce} has a variational structure
\begin{equation}
  \label{energyfunctional}
  E_f[u]=\frac{E[u]}{(\dashint_{S^n}fu^{2^*}~d\mu_{S^n})^{2/2^*}},
\end{equation}
where
$$E[u]=\dashint_{S^n}c_n|\nabla u|_{g_{S^n}}^2~d\mu_{S^n}+n(n-1)\dashint_{S^n}u^2~d\mu_{S^n}$$
Observe that if $u\in C^2$, then a simple integration by parts and \eqref{formulaforR} shows that
\begin{eqnarray}
  \label{altexpforE[u]}
  E[u]&=&\dashint_{S^n}[-c_n\Delta_{g_{S^n}}u+n(n-1)u]u~d\mu_{S^n}\nonumber\\
  &=&\dashint_{S^n}R~d\mu_g.
\end{eqnarray}
Hence, $E[u]$ is nothing but the average of total scalar curvature of metric $g$. Now, to achieve our goal, we find it is convenient to fix the volume of metric $g$ along the flow. That is
\begin{equation*}
  0=\frac{d}{dt}vol(g(t))=\frac{d}{dt}\int_{S^n}u^{2^*}~d\mu_{g_{S^n}}=\frac{n}{2}\int_{S^n}(\lambda(t)f-R)~d\mu_g,
\end{equation*}
which implies that
\begin{equation}
  \label{formulaforlambda}
  \lambda(t)=\frac{\int_{S^n}R~d\mu_g}{\int_{S^n}f~d\mu_g}=\frac{E[u]}{\dashint_{S^n}f~d\mu_g}.
\end{equation}
To end this section, we collect some useful formulae without giving the detailed calculation, since they have appeared in \cite{cx}.
\begin{lemma}\label{eeforE_f[u]&R}
~~\\
{\upshape(i)} Let $u$ be a smooth solution of \eqref{eeforu}. Then one has
$$\frac{d}{dt}E_f[u]=-\frac{n-2}{2}\bigg(\dashint_{S^n}f~d\mu_g\bigg)^{-\frac{2}{2^*}}\dashint_{S^n}(R-\lambda(t)f)^2~d\mu_g,$$
In particular, the energy functional $E_f[u]$ is decay along the flow.\\
{\upshape(ii)} The scalar curvature satisfies the evolution equation
\begin{equation}\label{eeforR}
\frac{\partial}{\partial t}(\lambda(t)f-R)=(n-1)\Delta_g(\lambda(t)f-R)+(\lambda(t)f-R)R+\lambda^\prime(t)f.
\end{equation}
\end{lemma}
\section{Global existence of the flow}
In this section, we focus on the global existence of the flow \eqref{eeforu}. The key point for this purpose is that our flow will preserve the positive property of the quantity $\int_{S^n}fu^{2^*}~d\mu_{S^n}$ when we initially have $\int_{S^n}fu_0^{2^*}~d\mu_{S^n}>0$.
\begin{lemma}
  \label{pospro}
  Assume that $\int_{S^n}fu_0^{2^*}~d\mu_{S^n}>0$ and $u$ is a smooth solution of the flow \eqref{eeforu} on $[0,T)$ for some $T>0$. Then one has
  $$\int_{S^n}fu(t)^{2^*}~d\mu_{S^n}>0,$$
  for all $t\in[0,T)$.
\end{lemma}
\begin{proof}
Since $\int_{S^n}fu_0^{2^*}~d\mu_{S^n}>0$, it follows from the definition of $E_f[u]$ that $E_f[u_0]>0$. By the Sharp Sobolev inequality and the volume-preserving property of our flow, we can obtain
\begin{equation}\label{lowerboundofE[u]}
 E[u]\geqslant n(n-1)\bigg(\dashint_{S^n}u_0^{2^*}~d\mu_{S^n}\bigg)^{2/2^*}.
 \end{equation}
 On the other hand, we have, by Lemma \ref{eeforE_f[u]&R} (i), that
 $$\frac{E[u]}{(\dashint_{S^n}fu(t)^{2^*}~d\mu_{S^n})^{2/2^*}}\leqslant E_f[u_0].$$
 Combining the two inequalities above yields
 \begin{equation}\label{poslowerbound}
\dashint_{S^n}fu(t)^{2^*}~d\mu_{S^n}\geqslant\bigg(\frac{n(n-1)}{E_f[u_0]}\bigg)^{2^*/2}\dashint_{S^n}u_0^{2^*}~d\mu_{S^n}>0.
\end{equation}
\end{proof}
With help of the Lemma above, we are able to show the boundedness of $\lambda(t)$.
\begin{lemma}
  \label{bdoflambda}
  During the evolution of the flow \eqref{eeforu}, $\lambda(t)$ remains bounded. More precisely, we have
  $$\lambda_1\leqslant \lambda(t)\leqslant\lambda_2,$$
  where
  $$\lambda_1=\frac{n(n-1)}{\max_{S^n}f}\bigg(\dashint_{S^n}u_0^{2^*}~d\mu_{S^n}\bigg)^{-\frac2n}\quad\mbox{and}\quad
  \lambda_2=\bigg(\frac{E_f[u_0]}{(n(n-1))^{2/n}}\bigg)^{\frac{n}{(n-2)}}\bigg(\dashint_{S^n}u_0^{2^*}~d\mu_{S^n}\bigg)^{-\frac2n}.$$
\end{lemma}
\begin{proof}
  In view of \eqref{formulaforlambda} and \eqref{energyfunctional}, we can rewrite $\lambda(t)$ as
  $$\lambda(t)=E_f[u]\bigg(\dashint_{S^n}fu^{2^*}~d\mu_{S^n}\bigg)^{-2/n}.$$
  On one hand, using \eqref{poslowerbound} and the non-increasing property of $E_f[u]$  we can get
  \begin{eqnarray*}
  \lambda(t)&\leqslant& E_f[u_0]\bigg(\frac{n(n-1)}{E_f[u_0]}\bigg)^{-\frac{2}{(n-2)}}\bigg(\dashint_{S^n}u_0^{2^*}~d\mu_{S^n}\bigg)^{-\frac2n}\\
  &=&\bigg(\frac{E_f[u_0]}{(n(n-1))^{2/n}}\bigg)^{\frac{n}{(n-2)}}\bigg(\dashint_{S^n}u_0^{2^*}~d\mu_{S^n}\bigg)^{-\frac2n}.
  \end{eqnarray*}
  On the other hand, it follows from \eqref{lowerboundofE[u]} and the volume-preserving property of our flow that
  \begin{eqnarray*}
    \lambda(t)&=&\frac{E[u]}{\dashint_{S^n}fu^{2^*}~d\mu_{S^n}}
    \geqslant\frac{n(n-1)\big(\dashint_{S^n}u_0^{2^*}~d\mu_{S^n}\big)^{2/2^*}}{(\max_{S^n}f)\dashint_{S^n}u_0^{2^*}~d\mu_{S^n}}\\
    &=&\frac{n(n-1)}{\max_{S^n}f}\bigg(\dashint_{S^n}u_0^{2^*}~d\mu_{S^n}\bigg)^{-\frac2n}.
  \end{eqnarray*}
  \end{proof}
  Now, we set
  $$F_2(t)=\dashint_{S^n}(\lambda f-R)^2~d\mu_g.$$
  We then have that the derivative of the normalized factor $\lambda(t)$ is bounded by $F_2(t)$.
  \begin{lemma}
    Let $u$ be a smooth solution of flow \eqref{eeforu}. Then there holds
    \begin{equation}\label{lambdaprime}
    \lambda^\prime(t)=-\bigg(\dashint_{S^n}f~d\mu_g\bigg)^{-1}\bigg[\frac{n-2}{2}\dashint_{S^n}(\lambda f-R)^2~d\mu_g+\dashint_{S^n}\lambda f(\lambda f-R)~d\mu_g\bigg].
    \end{equation}
    In particular,
    \begin{equation}\label{lbdatbound}
    |\lambda^\prime(t)|\leqslant C\Big(F_2(t)+\sqrt{F_2(t)}\Big),
    \end{equation}
    where $C>0$ is a universal constant.
  \end{lemma}
  \begin{proof}
    Since $\lambda(t)$ can be rewritten as
    $$\lambda(t)=E_f[u]\cdot\bigg(\dashint_{S^n}f~d\mu_g\bigg)^{-\frac2n},$$
    by using Lemma \ref{eeforE_f[u]&R}, \eqref{eeforu} and a direct computation, we can obtain \eqref{lambdaprime}. As for \eqref{lbdatbound}, it follows from the H\"older's inequality, Lemma \ref{bdoflambda} and \eqref{poslowerbound}.
  \end{proof}
  In order to get the uniform bound for $\lambda^\prime(t)$, we have to bound the quantity $F_2(t)$. However, we need to restrict the dimension of $S^n$ to be at least $4$.
  \begin{lemma}
    \label{boundofF_2}
    If $n\ge4$, then one can find a universal constant $C>0$ such that
    $$F_2(t)\leqslant C,$$
    for all $t\geqslant0$.
  \end{lemma}
  \begin{proof}
    Using Lemma \ref{eeforR} and \eqref{eeforu}, we have
    \begin{eqnarray*}
      \frac{d}{dt}F_2(t)&=&\frac{d}{dt}\dashint_{S^n}(\lambda f-R)^2~d\mu_g\\
      &=&2\dashint_{S^n}(\lambda f-R)\Big[(n-1)\Delta_g(\lambda f-R)+(\lambda f-R)R+\lambda^\prime f\Big]~d\mu_g\\
      &&+\frac{n}{2}\dashint_{S^n}(\lambda f-R)^3d\mu_g\\
      &=&-2(n-1)\dashint_{S^n}|\nabla(\lambda f-R)|_g^2~d\mu_g+\frac{4-n}{2}\dashint_{S^n}R(\lambda f-R)^2~d\mu_g\\
      &&+2\lambda^\prime\dashint_{S^n}f(\lambda f-R)~d\mu_g+\frac n2\dashint_{S^n}\lambda f(\lambda f-R)^2~d\mu_g\\
      &=&-\frac{n-4}{2}\bigg[\dashint_{S^n}c_n|\nabla(\lambda f-R)|_g^2~d\mu_g+\dashint_{S^n}R(\lambda f-R)^2~d\mu_g\bigg]\\
      &&-\dashint_{S^n}c_n|\nabla(\lambda f-R)|_g^2~d\mu_g+2\lambda^\prime\dashint_{S^n}f(\lambda f-R)~d\mu_g\\
      &&+\frac n2\dashint_{S^n}\lambda f(\lambda f-R)^2~d\mu_g,
    \end{eqnarray*}
    which implies by sharp Sobolev inequality, \eqref{lbdatbound}, and H\"older's inequality that
    \begin{eqnarray*}
      \frac{d}{dt}F_2(t)&\leqslant&-\frac{n(n-1)(n-4)}{2}\bigg(\dashint_{S^n}|\lambda f-R|^{2^*}~d\mu_g\bigg)^\frac{2}{2^*}-\dashint_{S^n}c_n|\nabla(\lambda f-R)|_g^2~d\mu_g\\
      &&+2\lambda^\prime\dashint_{S^n}f(\lambda f-R)~d\mu_g+\frac n2\dashint_{S^n}\lambda f(\lambda f-R)^2~d\mu_g\\
      &\leqslant&CF_2(t)\Big(1+\sqrt{F_2(t)}\Big).
    \end{eqnarray*}
    Set $$w(t)=\int_0^{F_2(t)}\frac{1}{1+\sqrt{s}}~ds.$$
    Then,
    \begin{equation}\label{dw/dt}
    \frac{dw}{dt}\leqslant CF_2(t).
    \end{equation}
    From Lemma \ref{eeforE_f[u]&R}, it follows that
    $$F_2(t)\leqslant-C\frac{d}{dt}E_f[u].$$
    Hence, by the fact that $E_f[u]\geqslant0$, we get
    $$\int_0^tF_2(s)~ds\leqslant CE_f[u_0].$$
    Now, integrating \eqref{dw/dt} from $0$ to $t$ with $t>0$ yields
    \begin{eqnarray}\label{upperboundofw}
      w(t)\leqslant w(0)+C\int_0^tF_2(s)~ds\leqslant F_2(0)+CE_f[u_0].
    \end{eqnarray}
    It is easy to see by the definition of $w(t)$ that
    \begin{equation}
      \label{lowerboundofw}
      w(t)\geqslant\frac{F_2(t)}{1+\sqrt{F_2(t)}}\geqslant
      \begin{cases}
        \frac{F_2(t)}{2}, & F_2(t)\leqslant 1,\\
        \frac{\sqrt{F_2(t)}}{2}, & F_2(t)>1.
      \end{cases}
    \end{equation}
    Combining \eqref{upperboundofw} and \eqref{lowerboundofw} yields the conclusion.
  \end{proof}
  Now, we immediately have the corollary
  \begin{corollary}
    \label{boundoflambda_t}
    There exists a universal constant $\Lambda_0>0$ such that
    $$
    \left\{
    \begin{array}{ll}
    \lambda^\prime(t)\leqslant\Lambda_0,&\mbox{if}~~n=3,\\[0.5em]
    |\lambda^\prime(t)|\leqslant\Lambda_0,&\mbox{if}~~n\geqslant4,
    \end{array}\right.
    $$
    for all $t\geqslant0$.
  \end{corollary}
 \begin{proof}
   If $n=3$, we apply the Young's inequality to \eqref{lambdaprime} and then use Lemma \ref{bdoflambda}, \eqref{poslowerbound} and the volume-preserving property of the flow \eqref{eeforu} to obtain
   \begin{eqnarray*}
     \lambda^\prime(t)&\leqslant&\bigg(\dashint_{S^n}f~d\mu_g\bigg)^{-1}\bigg[-\frac{n-2}{4}F_2(t)+\frac{1}{n-2}\dashint_{S^n}(\lambda|f|)^2~d\mu_g\bigg]\\
     &\leqslant& \frac{(\lambda_2\max_{S^n}|f|)^2}{n-2}\bigg(\frac{E_f[u_0]}{n(n-1)}\bigg)^{2^*/2}.
   \end{eqnarray*}
   For $n\geqslant4$, it immediately follows from \eqref{lbdatbound} and Lemma \ref{boundofF_2} that there exists a universal constant $\bar{C}>0$ such that $|\lambda^\prime(t)|\leqslant\bar{C}$. Now, by setting $$\Lambda_0=\max\Big\{\frac{(\lambda_2\max_{S^n}|f|)^2}{n-2}\bigg(\frac{E_f[u_0]}{n(n-1)}\bigg)^{2^*/2},\bar{C}\Big\},$$
   we thus complete the proof.
 \end{proof}
   Up to here, we are ready to apply the standard maximum principle to obtain a uniform lower bound for the scalar curvature $R$. We set
   $$C_0:=\min\Big\{-\sqrt{2(\lambda_2\max_{S^n}|f|)^2+2\Lambda_0\max_{S^n}|f|},~~\min_{S^n}R(0)-\lambda_2\max_{S^n}|f|\Big\}.$$
  \begin{lemma}\label{uniformboundofR}
    The scalar curvature function $R$ of $g$ satisfies
    $$R-\lambda(t)f\geqslant C_0,$$
    for all $t\geqslant0$.
  \end{lemma}
  \begin{proof}
    Define
    $$L=\partial_t-(n-1)\Delta_g+(\lambda f-C_0).$$
    By a simple calculation, \eqref{eeforR} and our choice of $C_0$, we can get
    \begin{eqnarray*}
      L(\lambda f-R+C_0)&=&\partial_t(\lambda f-R)-(n-1)\Delta_g(\lambda f-R)+(\lambda f-C_0)(\lambda f-R+C_0)\\
      &=&\lambda^\prime f+(\lambda f)^2-R^2-C_0^2+C_0R\\
      &\leqslant& \Lambda_0\max_{S^n}|f|+(\lambda_2\max_{S^n}|f|)^2-\frac{C_0^2}{2}\leqslant0.
    \end{eqnarray*}
    Moreover, it is easy to see that $\lambda f-C_0\geqslant0$ and $(\lambda f-R+C_0)(0)\leqslant0$ due to the choice of $C_0$. Hence, we can apply the maximum principle to operator $L$ to get $\lambda f-R+C_0\leqslant0$, which proves the assertion.
  \end{proof}
  Once we have the positivity-preserving property of $\int_{S^n}fu^{2^*}~d\mu_{S^n}$, the boundedness of $\lambda(t)$ and the uniform lower boundedness of $R$, we can follow exactly the same scheme in \cite{cx} to show that the flow \eqref{eeforu} can not blow up in finite time which is the following proposition.
  \begin{proposition}
    The flow \eqref{eeforu} has a unique smooth solution which is defined on $[0,+\infty)$.
  \end{proposition}
\qed
 \section{Blow-up analysis}
In this section, we dealt with the convergence of the flow \eqref{eeforu}. As an initial step, we notice the following $L^p$ convergence which is one of the key ingredients.
\begin{proposition}
  \label{lpconvergence}
  For $0<p<+\infty$ there holds
  $$\int_{S^n}|R-\lambda(t)f|^p~d\mu_g\rightarrow0,\qquad\mbox{as}~~t\rightarrow+\infty.$$
\end{proposition}
\begin{proof}
  Since the proof is exactly the same as in \cite[Lemma 3.2]{cx}, we omit it.
\end{proof}
\subsection{Compactness-Concentration}
Now, in order to prove the convergence, we have to bound the conformal factor $u$ uniformly. However, one, in general, can not realize this uniform bound directly. Here, we find the Compactness-Concentration theorem in \cite{ss} serving good purpose for us. Thus, we state this theorem whose proof can be found in \cite[Theorem 3.1]{ss}.
\begin{theorem}[Schwetlick \& Struwe]\label{cc}
Assume that $(M,g_0)$ is a compact Riemannian manifold without boundary. Let $g_k=u_k^{4/(n-2)}g_0$ with $0<u_k\in C^\infty(M,g_0)$ be a family of conformal metrics with unit volume and satisfying
\begin{equation}\label{integralboundcondition}
\int_{M}R_{g_k}~d\mu_{g_k}\leqslant C_1,\qquad \int_{M}|R_{g_k}-\int_{S^n}R_{g_k}~d\mu_{g_k}|^p~d\mu_{g_k}\leqslant C_1,
\end{equation}
for all $k\in\mathbb{N}$ and $p>\frac n2$. Then either\\
{\upshape(i)} the sequence $u_k$ is uniformly bounded in $W^{2,p}(M,g_0)$; or\\
{\upshape(ii)} there exists a subsequence $(u_k)_k$ (relabelled) and finitely many points\\
$x_1, x_2,\dots, x_m\in M$ such that for any $r>0$ and any $i\in\{1,\dots, m\}$ there holds
\begin{equation}
  \label{concentration}
  \liminf_{k\rightarrow+\infty}\bigg(\int_{B_r(x_i)}|R_{g_k}|^\frac{n}{2}~d\mu_{g_k}\bigg)^\frac 2n\geqslant n(n-1)\omega_n^\frac 2n,
\end{equation}
where $B_r(x_i)$ is the geodesic ball with radius $r$ and center $x_i$. Moreover, the sequence $(u_k)_k$ is bounded in $W^{2,p}$ on any compact subset of $(M\backslash\{x_1,x_2,\dots, x_m\},g_0)$.
\end{theorem}
\begin{remark}\label{rk1}
  Notice that the assumption of unit volume of $g_k$ in the theorem is not critical. In fact, only if the volume is uniformly bounded, one then has the same conclusion. Moreover, from the proof of the theorem, one can, in fact, conclude that $\max_{B_r(x_i)}u_k\rightarrow+\infty$ as $k\rightarrow+\infty$ for any $r>0$ and $i=1, \dots, m$.
\end{remark}
Now, we are ready to apply this theorem to our flow. Before doing so, let us set up some notations. Choose an arbitrary time sequence $(t_k)_k\subset[0,+\infty)$ with $t_k\rightarrow+\infty$ as $k\rightarrow+\infty$. We set
$$u_k=u(t_k),\quad g_k=g(t_k),\quad R_k=R_{g_k}\quad\mbox{and}\quad d\mu_k=d\mu_{g_k}.$$
\begin{lemma}
  \label{compactness}
 If $u_k$ is bounded in $W^{2,p}(S^n,g_{S^n})$ for $p>\frac n2$, then there exists $0<u_\infty\in W^{2,p}(S^n,g_{S^n})$ such that $u_k\rightarrow u_\infty$ in $W^{2,p}(S^n,g_{S^n})$ as $k\rightarrow+\infty$. In addition, if we let $g_\infty=u_\infty^{4/(n-2)}g_{S^n}$, then $g_\infty$, up to a constant multiple, has the scalar curvature $f$.
\end{lemma}
\begin{proof}
  Since the proof is rather standard, we omit the detail. The reader can also refer to \cite[Lemma 4.5]{cx}.
\end{proof}
It follows from Proposition \ref{lpconvergence} and Minkowski's inequality that
\begin{equation}
  \label{lpscalarcurvaturebound}
  \int_{S^n}R_{g_k}~d\mu_{g_k}\leqslant C_1,\qquad \int_{S^n}|R_{g_k}-\int_{S^n}R_{g_k}~d\mu_{g_k}|^p~d\mu_{g_k}\leqslant C_1,
\end{equation}
for any $p>n/2$. Hence, condition \eqref{integralboundcondition} holds true for $g_k$. In addition, our flow preserves the volume which implies that $vol(g_k)=vol(g_0)$. In particular, volume of $g_k$ is uniformly bounded. Hence, by Remark \ref{rk1} we can apply Theorem \ref{cc} to this sequential metrics $(g_k)_k$. If the case (i) in Theorem \ref{cc} happens to $g_k$, i.e., there exists a uniform positive constant $C$ such that $||u_k||_{W^{2,p}}\leqslant C$, then Lemma \ref{compactness} shows that the prescribed function $f$ can be the scalar curvature of some conformal metric $g$. Therefore, to prove our theorems, it suffices to show that there exists some time sequence $t_k$ such that $||u_k||_{W^{2,p}}\leqslant C$. However, it is hard, in general, to realize this directly. Here, we adopt the contradiction argument. We assume that $f$ can not be realized as the scalar curvature of any conformal metric on $S^n$. This means that for arbitrary time sequence $t_k$ with the corresponding sequential metric $(g_k)_k$, the case (ii) will occur. Hence, one can expect that the blow-up phenomenon will appear and the blow-up analysis has to come into play.
\subsection{Blow-up analysis}\label{blowup}
In this subsection, we mainly perform the blow-up analysis for the sequential metrics $(g_k)_k$. In the proceeding proof, we always assume that the second case in Theorem \ref{cc} occurs to $(g_k)_k$. The key step for deriving the blow-up behavior is to uniformly bound the normalized function $v$ which is defined in \eqref{formulaforv} below. Here, our method is different from that in \cite{cx}. Roughly speaking, Chen \& Xu's method is to estimate the first eigenvalue of $g_k$, and then apply higher dimensional version of Proposition A in \cite{cy4} to claim that $v_k:=v(\cdot,t_k)$ is bounded in $L^{p_0}$ for some $p_0$ slightly greater than $2^*$. Finally, by an iteration argument, $v_k$ is indeed bounded in $L^p$ for some $p>n$. Once this holds, Sobolev embedding theory immediately implies that $v_k$ is uniformly bounded. While our method, inspired by Struwe \cite{str}, is to apply Theorem \ref{cc} to the corresponding sequential normalized metrics $(h_k)_k$ (see the exact definition below). Then, under the assumption of cases (ii) in Theorem \ref{cc} occurring to $(h_k)_k$, we are managed to estimate the center of mass of $h_k$: $\int_{S^n}x~d\mu_{h_k}$ and show that $\int_{S^n}x~d\mu_{h_k}\neq0$ for large $k$ which violates the normalized condition \eqref{normalizedcondition}. Hence, case (ii) cannot occur. In other words, case (i) in Theorem \ref{cc} will happen. But in this case, Sobolev embedding theory shows that $v_k$ is uniformly bounded. Notice that, in the proof, the condition (ii) in Theorem \ref{main} and \ref{main1} will play a crucial role.

Up to here, let us define the so-called normalized flow. It is a well known fact that, for every smoothly varying family of metrics $g(t)=u(t)^{4/(n-2)}g_{S^n}$, there exists a family of conformal transformations $\varphi(t):S^n\mapsto S^n$ such that
\begin{equation}
  \label{normalizedcondition}
  \int_{S^n}x~d\mu_h=0,\qquad\mbox{for}~~t\ge0,
\end{equation}
where $h=\varphi^*g$ and $x=(x^1, x^2,\dots, x^{n+1})$. In fact, let $\pi:S^n\backslash\{0, 0,\dots, 0, -1\}\mapsto\mathbb{R}^n$ be the stereographic projection from the south pole to $n$-plane and set, for fixed $q\in\mathbb{R}^n, r>0$, $\delta_{q,r}(z)=q+rz$ for $z\in\mathbb{R}^n$. Then $\varphi(t)$ can be written as
$$\varphi(t)=\pi^{-1}\circ\delta_{q(t),r(t)}\circ\pi.$$
The pullback metric $h=\varphi^*g$ is called the normalized metric. In terms of $u$, it can be written as $h=v^{4/(n-2)}g_{S^n}$, where
\begin{equation}
  \label{formulaforv}
  v=(u\circ\varphi)|\det(d\varphi)|^{\frac{n-2}{2n}},
\end{equation}
which satisfies the equation
\begin{equation}
  \label{equationofv}
  -c_n\Delta_{S^n}v+n(n-1)v=R_hv^{2^*-1},
\end{equation}
where $R_h=R_g\circ\varphi(t)$. Differentiating \eqref{formulaforv} w.r.t to $t$, we obtain the evolution equation
\begin{equation*}
  v_t=(u_t\circ\varphi)|\det(d\varphi)|^{\frac{n-2}{2n}}+\frac{n-2}{2n}v^{1-2^*}\mbox{div}_{S^n}(v^{2^*}\xi),
\end{equation*}
where $\xi=(d\varphi)^{-1}(d\varphi/dt)$ is the vector field on $S^n$.

Before continuing our argument, we want to point out that the behavior of $q(t)$ and $r(t)$ plays an important role in the blow-up analysis. Here, for sake of simplifying the calculation, we will follow the idea in \cite{ms}. That is, for each fixed $t_0\geqslant0$, we make a translation and a scale such that $q(t_0)=0$ and $r(t_0)=1$. Precisely speaking, as in \cite{ms}, for $t_0\geqslant 0$ fixed and $t\geqslant0$ close to $t_0$, let
$$\varphi_{t_0}(t)=\varphi(t_0)^{-1}\varphi(t).$$
Then
$$\varphi_{t_0}(t)\circ\psi=\psi_{q(t),r(t)},$$
where $\psi=\pi^{-1}$ and $\psi_{q,r}=\psi\circ\delta_{q,r}$. 

Now, given $t_0\geqslant0$, we consider a rotation mapping some $p={p}(t_0)\in(S^n)$ into the north pole $N=(0, 0, \dots, 1)$. Then $\varphi(t_0)$ can be expressed as $\varphi(t_0)=\psi_\epsilon\circ\pi$ for some $\epsilon=\epsilon(t_0)>0$, where $\psi_{\epsilon}(z)=\psi(\epsilon z)=\psi_{0,\epsilon}(z)$ by the notation above. Hence, in stereographic coordinates, $\varphi(t):=\varphi_{p(t),\epsilon(t)}$ is given by
$$\varphi(t)\circ\psi=\varphi(t_0)\circ\varphi_{t_0}(t)\circ\psi=\psi_\epsilon\circ\delta_{q,r}.$$

So, in the following, our calculations, involving any conformal transformation, are always at the each fixed time $t_0$. In this way, the conformal transformation has the expression: $\varphi_{p,\epsilon}=\psi_\epsilon\circ\pi$. However, we want to abuse the notation a bit to use the parameter $t$ instead of $t_0$ in the computation.

Also let us define
$$X_*=\bigg\{0<u\in C^\infty(S^n);\int_{S^n}u^{2^*}~d\mu_{S^n}=\omega_n\bigg\}.$$
Recall that $f$ satisfies $(\max_{S^n}|f|)\big/(\dashint_{S^n}f~d\mu_{S^n})<2^\frac 2n$. Hence, we can choose
$$\sigma=\frac12\Bigg[2^\frac2n\bigg(\frac{\dashint_{S^n}f~d\mu_{S^n}}{\max_{S^n}|f|}\bigg)-1\Bigg]>0,
$$
and set
\begin{equation}
  \label{gamma}
  \gamma=n(n-1)\big[(1+\sigma)\big]^\frac{n-2}{n}\bigg(\dashint_{S^n}f~d\mu_{S^n}\bigg)^\frac{2-n}{n}.
\end{equation}
With all notations above settled, we finally define the set
$$X_f=\bigg\{u\in X_*; \dashint_{S^n}fu^{2^*}~d\mu_{S^n}>0~~\mbox{and}~~E_f[u]\leqslant\gamma\bigg\}.$$
\begin{remark}
  \label{rk2}
  Notice that $X_f$ is not an empty set. In fact, when $u\equiv1$ we have $\int_{S^n}u^{2^*}=\omega_n$, $\dashint_{S^n}fu^{2^*}~d\mu_{S^n}=\dashint_{S^n}f~d\mu_{S^n}>0$ and $E_f[u]=n(n-1)(\dashint_{S^n}f~d\mu_{S^n})^{\frac{2-n}{n}}<\gamma$. Hence, $u\equiv1\in X_f$.
\end{remark}
\begin{proposition}
  \label{blowupbehavior}
  Let $u(t)$ be the smooth solution of the flow \eqref{eeforu} with initial data $u_0\in X_f$. Associated with the sequential metrics $g_k=u_k^{4/(n-2)}g_{S^n}$, we let $h_k=\varphi_k^*g_k=v_k^{4/(n-2)}g_{S^n}$ be the sequence of corresponding normalized metrics defined above, where $\varphi_k=\varphi_{p_k,\epsilon_k}$ with $p_k=p(t_k)$ and $\epsilon_k=\epsilon(t_k)$. Then, one has\\
  {\upshape(i)} There exists only one point $Q\in S^n$ such that concentration phenomenon in the sense of \eqref{concentration} can occur; up to a subsequence, there holds\\
  {\upshape(ii)} For any $r>0$, $\max_{B_r(Q)}u_k\rightarrow+\infty$ as $k\rightarrow+\infty$;\\
  {\upshape(iii)} $||v_k-1||_{C^{1,\alpha}}\rightarrow0$ as $k\rightarrow+\infty$ for $\alpha\in(0,1)$,\\
  {\upshape(iv)} $d\mu_k\rightarrow\omega_n\delta_Q$ weakly in the sense of measure, where $\delta_Q$ is the dirac measure and\\
  {\upshape(v)} $\varphi_k\rightarrow Q$ for almost every $x\in S^n$, and there exists a constant $\lambda_\infty\in[\lambda_1,\lambda_2]$ such that $\lambda_k\rightarrow\lambda_\infty$ with $\lambda_\infty f(Q)=n(n-1)$. In particular, $f(Q)>0$.
\end{proposition}
\begin{proof}
  (i). It follows from Lemma \ref{bdoflambda} and the choices of initial data $u_0$ and $\gamma$ that
  \begin{eqnarray}\label{lbdabound}
    \lambda(t_k)&\leqslant&\bigg[\frac{E_f[u_0]}{(n(n-1))^\frac 2n}\bigg]^\frac{n}{n-2}\nonumber\\
    &\leqslant&\Bigg[\frac{n(n-1)(1+\sigma)^\frac{n-2}{n}\Big(\dashint_{S^n}f~d\mu_{S^n}\Big)^\frac{2-n}{n}}{(n(n-1))^\frac 2n}\Bigg]^\frac{n}{n-2}\nonumber\\
    &=&n(n-1)(1+\sigma)\bigg(\dashint_{S^n}f~d\mu_{S^n}\bigg)^{-1},
  \end{eqnarray}
  which implies that
    \begin{eqnarray}\label{lbdafbound}
    \lambda(t_k)\bigg(\dashint_{S^n}|f|^\frac n2~d\mu_k\bigg)^\frac2n&\leqslant&
    n(n-1)(1+\sigma)\frac{\max_{S^n}|f|}{\dashint_{S^n}f~d\mu_{S^n}}.
  \end{eqnarray}
  From Proposition \ref{lpconvergence}, the estimates \eqref{lbdafbound} and the condition (ii) in Theorems \ref{main} and \ref{main1}, we can get the estimate
  \begin{eqnarray*}
    &&\liminf_{k\rightarrow+\infty}\bigg(\dashint_{S^n}|R_k|^\frac n2~d\mu_k\bigg)^\frac2n\\
    &&\quad\leqslant\liminf_{k\rightarrow+\infty}\bigg[\bigg(\dashint_{S^n}|R_k-\lambda(t_k)f|^\frac n2~d\mu_k\bigg)^\frac2n+\lambda(t_k)\bigg(\dashint_{S^n}|f|^\frac n2~d\mu_k\bigg)^\frac2n\bigg]\\
    &&\quad=\liminf_{k\rightarrow+\infty}\lambda(t_k)\bigg(\dashint_{S^n}|f|^\frac n2~d\mu_k\bigg)^\frac2n\\
    &&\quad\leqslant n(n-1)(1+\sigma)\frac{\max_{S^n}|f|}{\dashint_{S^n}f~d\mu_{S^n}}\\
    &&\quad<n(n-1)2^\frac2n.
  \end{eqnarray*}
  Now, suppose $\{x_1, \dots, x_m \}$, defined in the Theorem \ref{cc}, are concentration points with $m\geqslant2$. Let $0<r<\frac12\{\mbox{dist}(x_i,x_j);1\leqslant i<j\leqslant m\}$. It follows from \eqref{concentration} and the estimate above that
  \begin{eqnarray*}
    m[n(n-1)]^\frac n2&\leqslant&\sum_{i=1}^m\liminf_{k\rightarrow+\infty}\omega_n^{-1}\int_{B_r(x_i)}|R_k|^\frac n2~d\mu_k\\
    &\leqslant&\liminf_{k\rightarrow+\infty}\bigg[\sum_{i=1}^m\omega_n^{-1}\int_{B_r(x_i)}|R_k|^\frac n2~d\mu_k\bigg]\\
    &\leqslant&\liminf_{k\rightarrow+\infty}\omega_n^{-1}\int_{\cup_{i=1}^mB_r(x_i)}|R_k|^\frac n2~d\mu_k\\
    &\leqslant&\liminf_{k\rightarrow+\infty}\dashint_{S^n}|R_k|^\frac n2~d\mu_k<2[n(n-1)]^\frac n2,
  \end{eqnarray*}
  which implies that $m<2$ and thus contradicts with $m\geqslant2$. This shows that $m=1$, i.e. concentration point is unique.\vspace{0.3em}
  
  \noindent(ii) It follows directly from Remark \ref{rk1}.\vspace{0.3em}
  
  \noindent(iii) The proof will consist of several Claims below.

  {\bf Claim 1}: There exists a uniform constant $C>0$ such that $C^{-1}\leqslant v_k\leqslant C.$\\
  {\itshape Proof of Claim 1}: For normalized sequential metrics $(h_k)_k$, we note that
  $$\int_{S^n}\bigg|R_k\circ\varphi_k-\int_{S^n}R_k\circ\varphi_k~d\mu_{h_k}\bigg|^p~d\mu_k=\int_{S^n}\bigg|R_k-\int_{S^n}R_k~d\mu_k\bigg|^p~d\mu_{h_k},$$
  $$\int_{S^n}R_k\circ\varphi_k~d\mu_{h_k}=\int_{S^n}R_k~d\mu_k~~\mbox{and}~~vol(S^n,h_k)=vol(S^n,g_k)=\omega_n.$$
  Therefore, it is easy to see by \eqref{lpscalarcurvaturebound} that all the conditions in Theorem \ref{cc} hold for $(h_k)_k$. Hence, we can apply Theorem \ref{cc} to the sequence $(h_k)_k$. It means that there also have two alternatives for $(h_k)_k$. Now, if the second case in Theorem \ref{cc} happens to $(h_k)_k$, then we can follow the exact same proof of (i) to conclude that there exists the unique point $Q$ such that \eqref{concentration} holds for $(h_k)_k$. Hence, for sufficiently large $k$ and any $r>0$, we have by Proposition \ref{lpconvergence} and  \eqref{lbdabound} that
  \begin{eqnarray*}
    &&n(n-1)\omega_n^\frac 2n+o(1)\leqslant\bigg(\int_{B_r(Q)}|R_{h_k}|^\frac n2~d\mu_{h_k}\bigg)^\frac 2n\\
    &&\quad\leqslant\bigg(\int_{B_r(Q)}|R_{h_k}-\lambda(t_k)f\circ\varphi_k|^\frac n2~d\mu_{h_k}\bigg)^\frac 2n+\lambda(t_k)\max_{S^n}|f|\bigg(\int_{B_r(Q)}~d\mu_{h_k}\bigg)^\frac 2n\\
    &&\quad=o(1)+n(n-1)(1+\sigma)\frac{\max_{S^n}|f|}{\dashint_{S^n}f~d\mu_{S^n}}\bigg(\int_{B_r(Q)}~d\mu_{h_k}\bigg)^\frac 2n,
  \end{eqnarray*}
  which implies that
  \begin{equation*}
    \int_{B_r(Q)}~d\mu_{h_k}\geqslant \tau^{-\frac n2}\omega_n+o(1),\quad\mbox{where}~~\tau=(1+\sigma)\frac{\max_{S^n}|f|}{\dashint_{S^n}f~d\mu_{S^n}}.
  \end{equation*}
  Since $vol(S^n,h_k)=\omega_n$, we get
  \begin{equation}\label{v1}
    \int_{S^n\backslash B_r(Q)}~d\mu_{h_k}\leqslant(1-\tau^{-\frac n2})\omega_n+o(1).
  \end{equation}
  From \eqref{v1}, it follows that
  \begin{eqnarray*}
    \bigg|\Big|\dashint_{S^n}x~d\mu_{h_k}\Big|-1\bigg|&=&\bigg|\Big|\dashint_{S^n}x~d\mu_{h_k}\Big|-|Q|\bigg|\\
    &\leqslant&\bigg|\dashint_{S^n}x~d\mu_{h_k}-Q\bigg|\leqslant\dashint_{S^n}|x-Q|~d\mu_{h_k}\\
    &=&\omega_n^{-1}\int_{S^n\backslash B_r(Q)}|x-Q|~d\mu_{h_k}+\omega_n^{-1}\int_{B_r(Q)}|x-Q|~d\mu_{h_k}\\
    &\leqslant&2(1-\tau^{-\frac n2})+r+o(1).
  \end{eqnarray*}
  Notice that $\tau<2^{2/n}$. This implies that $2\tau^{-n/2}-1>0$. Now, by choosing $r=(2\tau^{-n/2}-1)/2$ and $k$ large enough, we then have
  \begin{eqnarray*}
    \bigg|\dashint_{S^n}x~d\mu_{h_k}\bigg|&\geqslant&1-2(1-\tau^{-\frac n2})-r+o(1)\\
    &=&\frac{2\tau^{-\frac n2}-1}{2}+o(1)>0.
  \end{eqnarray*}
  However, this contradicts with the fact that $h_k$ satisfies \eqref{normalizedcondition}. Such a contradiction shows that the second case can not happen to $(h_k)_k$. In other words, the first case in Theorem \ref{cc} will happen, that is, $v_k$ is uniformly bounded in $W^{2,p}(S^n,g_{S^n})$ for $p>\frac n2$. By Sobolev embedding theory, we conclude that there exists a positive constant $C$ such that $||v_k||_{C^{1,\alpha}(S^n)}\leqslant C$ for $\alpha=1-n/2p$. Let $$P(x):=n(n-1)+\sup_{k\in\mathbb{N}}\sup_{S^n}[-(\lambda(t_k)f\circ\varphi_k+C_0)v_k^{4/(n-2)}].$$
  Then it is easy to see that $P(x)$ is bounded. Moreover, by $v_k>0$ and Lemma \ref{uniformboundofR}, we have
  \begin{eqnarray*}
    0&\leqslant&(R_{h_k}-\lambda(t_k)f\circ\varphi_k-C_0)v_k^{2^*-1}\\
    &=&-c_n\Delta_{S^n}v_k+n(n-1)v_k-(\lambda(t_k)f\circ\varphi_k+C_0)v_k^{2^*-1}\\
    &\leqslant&-c_n\Delta_{S^n}v_k+P(x)v_k.
  \end{eqnarray*}
  Using \cite[Corollary A.3]{br1} and the fact that $\int_{S^n}v_k^{2^*}~d\mu_{g_{S^n}}=\omega_n$, we conclude that $v_k\geqslant C^{-1}$. This finishes the proof of Claim 1.

  {\bf Claim 2}: $v_k\rightarrow 1$ in $C^{1,\alpha}(S^n)$ with $\alpha\in(0,1)$.\\
  {\itshape Proof of Claim 2}. By the proof of Claim 1, we know that $v_k$ is uniformly bounded in $W^{2,p}(S^n,g_{S^n})$ for any $p>\frac n2$. Hence, Sobolev embedding theory shows that there exists $v_\infty\in C^{1,\alpha}(S^n)$ with $\alpha<1-n/2p$ such that, up to a subsequence,
  \begin{equation}\label{vkconvergence}
  v_k\rightarrow v_\infty~~\mbox{in}~~C^{1,\alpha}(S^n)\qquad\mbox{as}~~k\rightarrow+\infty.
  \end{equation}
 Moreover, since the conclusion above holds for any $p>n/2$, we get that it holds for $\alpha\in (0,1)$.
 Now, as $(p_k)_k\subset S^n$, we may assume that $p_k\rightarrow Q_*$ for some $Q_*\in S^n$. Notice that, in our convention, $Q_*$ is, in fact at the north pole of $S^n$(denoted by $\mathrm{N})$, since, for each $k$, we have made a rotation to put $p_k$ at $\mathrm{N}$. Moreover, by Lemma \ref{bdoflambda}, we assume that $\lambda_k\rightarrow\lambda_\infty$ for some $\lambda_\infty\in[\lambda_1,\lambda_2]$. We then claim that $\epsilon_k\rightarrow0$ as $k\rightarrow+\infty$. If not, we may suppose that $\epsilon_k\rightarrow \epsilon_0>0$, then $\varphi_k\rightarrow\varphi_{Q_*,\epsilon_0}$ as $k\rightarrow+\infty$, and thus $\det(d\varphi_k)\rightarrow \det(d\varphi_{Q_*,\epsilon_0})$ which is bounded away from zero. From this and the proved fact that $C^{-1}\leqslant v_k\leqslant C$, it follows, by \eqref{formulaforv}, that $u_k$ is bounded from below and above by a positive constant. In view of the equation
 $$-c_n\Delta_{S^n}u_k+n(n-1)u_k=R_ku_k^{2^*-1},$$
 we can conclude that $u_k$ is uniformly bounded in $W^{2,p}(S^n,g_{S^n})$, which contradicts our assumption. Hence, we have $\epsilon_k\rightarrow0$. By the definition of $\varphi_k$, we conclude that $\varphi_k\rightarrow Q_*$ for $x\in S^n\backslash\{-Q_*\}$, which together with Lemma \ref{lpconvergence}, Claim 1 and the dominated convergence theorem imply that
 \begin{eqnarray}
   \label{R_klpconvergence}
   &&||R_{h_k}-\lambda_\infty f(Q_*)||_{L^p(S^n,h_k)}\nonumber\\
   &&\quad\leqslant||R_{h_k}-\lambda(t_k)f\circ\varphi_k||_{L^p(S^n,h_k)}+|\lambda(t_k)-\lambda_\infty|||f\circ\varphi_k||_{L^p(S^n,h_k)}\nonumber\\
   &&\qquad+\lambda_\infty||f\circ\varphi_k-f(Q_*)||_{L^p(S^n,h_k)}\rightarrow0,
 \end{eqnarray}
  as $k\rightarrow+\infty$. This implies that $v_\infty$ weakly solves
  \begin{equation*}
    -c_n\Delta_{S^n}v_\infty+n(n-1)v_\infty=\lambda_\infty f(Q_*)v_\infty^{2^*-1}.
  \end{equation*}
  Since we have $\int_{S^n}xv_k^{2^*}~d\mu_{S^n}$=0 and $\int_{S^n}v_k^{2^*}~d\mu_{S^n}=\omega_n$, it follows, by \eqref{vkconvergence}, that $v_\infty$ will satisfy $\int_{S^n}xv_\infty^{2^*}~d\mu_{S^n}$=0 and $\int_{S^n}v_\infty^{2^*}~d\mu_{S^n}=\omega_n$. By the classification theorem, we conclude that $v_\infty$ must be a constant and $v_\infty\equiv1$. Moreover, plugging $v_\infty$ into the equation above yields $\lambda_\infty f(Q_*)=n(n-1)$.

  \noindent{(iii)} Since $\lambda_\infty f(Q_*)=n(n-1)$, it follows from \eqref{R_klpconvergence} that
  $$||R_k-n(n-1)||_{L^p(S^n,g_k)}\rightarrow0,\quad\mbox{as}~~k\rightarrow+\infty.$$
  Hence, for large $k$ and any $r>0$, we have by conclusion (i) in this proposition that
  \begin{eqnarray*}
   &&n(n-1)\omega_n^\frac 2n+o(1)\leqslant\bigg(\int_{B_r(Q)}|R_k|^\frac n2~d\mu_{k}\bigg)^\frac 2n\\
    &&\quad\leqslant\bigg(\int_{B_r(Q)}|R_k-n(n-1)|^\frac n2~d\mu_{k}\bigg)^\frac 2n+n(n-1)\bigg(\int_{B_r(Q)}~d\mu_{k}\bigg)^\frac 2n\\
    &&\quad=o(1)+n(n-1)\bigg(\int_{B_r(Q)}~d\mu_{k}\bigg)^\frac 2n\leqslant n(n-1)\omega_n^\frac 2n+o(1),
  \end{eqnarray*}
  which concludes that
  $$d\mu_k\rightarrow\omega_n\delta_Q,\quad\mbox{as}~~k\rightarrow+\infty,$$
  in the sense of measure.

  \noindent(iv) In view of the proof of (ii), we only need to show that $Q_*=Q$. In fact, on one hand, it follows from (iii) that
  $$\dashint_{S^n}x~d\mu_k\rightarrow Q,\quad\mbox{as}~~k\rightarrow+\infty.$$
  On the other hand, it follows from the fact that $v_k$ is uniformly bounded and the dominated convergence theorem that 
  \begin{eqnarray*}
    \bigg|\dashint_{S^n}\varphi_kd\mu_{h_k}-Q_*\bigg|&\leqslant&\dashint_{S^n}|\varphi_k-Q_*|d\mu_{h_k}\rightarrow0,
  \end{eqnarray*}
  as $k\rightarrow+\infty$. Notice that, by the change of variables, one has
  $$\dashint_{S^n}x~d\mu_k=\dashint_{S^n}\varphi_kd\mu_{h_k}.$$
  Hence, it is now easy to see that $Q_*=Q$.
  \end{proof}
\subsection{Asymptotic behavior of the flow} In this subsection, we will apply Proposition \ref{blowupbehavior} to study the asymptotic behavior of the flow $(u(t))$ in case of divergence. For $t\ge0$, let
$$S=S(t)=\dashint_{S^n}x~d\mu_g$$
be the center of mass of $g=g(t)$. Then we have
\begin{lemma}\label{S(t)not0}
  $S(t)\rightarrow Q$ as $t\rightarrow+\infty$. In particular, $S(t)\neq0$ for all large $t$.
\end{lemma}
\begin{proof}
  For an arbitrary time sequence $(t_k)_k\subset[0,+\infty)$ with $t_k\rightarrow+\infty$ as $k\rightarrow+\infty$, we can apply Proposition \ref{blowupbehavior} to the sequential metrics $(g(t_k))_k$ to get that $d\mu_{g_k}\rightarrow\delta_Q$ as $k\rightarrow+\infty$. Hence
  $S(t_k)\rightarrow Q$ as $k\rightarrow+\infty$. By the arbitrariness of the sequence $(t_k)_k$, we thus conclude that $S(t)\rightarrow Q$ as $t\rightarrow+\infty$.
\end{proof}
\noindent So, we may assume, w.l.o.g., that $S(t)\neq0$ for all $t\geqslant0$. Then the image of $S(t)$ under radial projection
$$Q(t)=S/|S|\in S^n$$
is well defined for all $t\geqslant0$.
Up to here, we are ready to describe the precise asymptotic behavior of the flow $(u(t))$.
\begin{proposition}\label{asymptoticbehavior}
  Suppose that $f$ can not be realized as the scalar curvature of any conformal metric on $S^n$. Let $u(t)$ be the smooth solution of \eqref{eeforu}, $\varphi(t)$ the conformal transformation, $v(t)$ the corresponding normalized flow and $h(t)$ the normalized metric. Then, as $t\rightarrow+\infty$, there hold
  \begin{itemize}
    \item[(i)] $\max_{S^n}u(\cdot,t)\rightarrow+\infty$,\\
    \item[(ii)] $v(t)\rightarrow1, h(t)\rightarrow g_{S^n}$ in $C^{1,\alpha}(S^n)$ for $\alpha\in(0,1)$,\\
    \item[(iii)] $\varphi(t)\rightarrow Q$ in $L^{2}(S^n,g_{S^n})$, $ Q(t)\rightarrow Q$ and $\lambda(t)f(Q(t))\rightarrow n(n-1)$,\\
    \item[(iv)] $E_f[u(t)]\rightarrow n(n-1)(f(Q))^\frac{2-n}{n}$, and moreover, one has\\
    \item[(v)] $\nabla_{S^n}f(Q)=0$ and $\Delta_{S^n}f(Q)<0$.
  \end{itemize}
\end{proposition}
\begin{proof}
  (i) (ii), (iii) and (iv) follow directly from Proposition \ref{blowupbehavior}, Lemma \ref{S(t)not0} and a contradiction argument, while (v) follows from exactly the same proof as in \cite[Proposition 6.1]{cx}.
\end{proof}
\section{Proof of Theorems \ref{main} and \ref{main1}}
In this section, we devote ourselves to proving the main results in this paper.
\subsection{Proof of Theorem \ref{main}}
For $p\in S^n, 0<\epsilon<+\infty$, if we put the point $p$ at the origin in stereographic coordinates, then by the notation we used in subsection \ref{blowup} we have $\varphi_{p,\epsilon}=\psi_\epsilon\circ\pi$. Let $g_{p,\epsilon}=u_{p,\epsilon}^{2^*}g_{S^n}$ with $u_{p,\epsilon}=|\det(d\varphi_{p,\epsilon})|^{1/2^*}$. Then
$$d\mu_{g_{p,\epsilon}}\rightarrow\omega_n\delta_p,\qquad\mbox{as}~~\epsilon\rightarrow0.$$

For $\rho\in\mathbb{R}_+$, denote the sub-level set of $E_f$ by
$$L_\rho=\{u\in X_f: E_f[u]\leqslant\rho\}.$$
By Proposition \ref{asymptoticbehavior}, we know that the concentration phenomenon can only occur at the critical points of $f$ where $f$ takes positive values. Hence, we label all positive critical points $p_1, \dots, p_N$ of $f$ so that $0<f(p_i)\leqslant f(p_j)$ for $1\leqslant i\leqslant j\leqslant N$ and let
$$\gamma_i=\frac{n(n-1)}{(f(p_i))^\frac{n-2}{n}}=\lim_{\epsilon\rightarrow0}E_f[u_{p_i,\epsilon}], 1\leqslant i\leqslant N.$$
For sake of convenience, we assume, w.l.o.g., that all positive critical levels $f(p_i)$, $1\leqslant i\leqslant N$ are distinct. By choosing $s_0=\frac{1}{3}\min_{i\leqslant i\leqslant N-1}\{\gamma_i-\gamma_{i+1}\}>0$, we then have $\gamma_i-2s_0>\gamma_{i+1}$ for all $i$.

With all notations set up, we try to describe, by Proposition \ref{homotopy} below, the homotopy on $L_\rho$. We remark here that the idea of such a homotopy result originally due to Malchiodi \& Struwe \cite{ms}. Chen \& Xu's argument in \cite{cx} still holds for (ii), (iii) and (iv) even if $f$ changes sign. However, the tricky homotopy mapping for the proof of (i), given by Chen \& Xu, does not work here anymore. The reason is that their mapping involves the conformal transformation $\varphi_s$ and so the term $\dashint_{S^n}f\circ\varphi_s~d\mu_{S^n}$ will appear. However, if this case, on one hand, we cannot guarantee the positivity of $\dashint_{S^n}f\circ\varphi_s~d\mu_{S^n}$; On the other hand, we can not compare $\dashint_{S^n}f\circ\varphi_s~d\mu_{S^n}$ and $\dashint_{S^n}f~d\mu_{S^n}$, and then, we can not control the energy level of $E_f[\varphi_s]$. Our homotpy mapping avoids to using any conformal transformation.
\begin{proposition}\label{homotopy}
~~\\
  {\upshape(i)} If $\max\Big\{\gamma_1,n(n-1)(\dashint_{S^n}f~d\mu_{S^n})^{(2-n)/n}\Big\}<\gamma_0\leqslant\gamma$, where $\gamma$ has been chosen in \eqref{gamma}, then $L_{\gamma_0}$ is contractible.\\
  {\upshape(ii)} For $0<s\leqslant s_0$ and each $i$, the set $L_{\gamma_i-s}$ is homotopy equivalent to the set $L_{\gamma_{i+1}+s}$.\\
  {\upshape(iii)} For each critical point $p_i$ of $f$ with $\Delta_{S^n}f(p_i)>0$, the set $L_{\gamma_i+s_0}$ is homotopy equivalent to the set $L_{\gamma_i-s_0}$.\\
  {\upshape(iv)} For each critical point $p_i$ of $f$ with $\Delta_{S^n}f(p_i)<0$, the set $L_{\gamma_i+s_0}$ is homotopy equivalent to the set $L_{\gamma_i-s_0}$ with $(n-\mbox{ind}_f(p_i))$-cell attached.
\end{proposition}
\begin{proof}
  (i). Notice that by Proposition \ref{asymptoticbehavior} we have the fact that for each $u_0\in X_f$,
  \begin{equation}\label{ugoestoinfty}
  \lim_{t\rightarrow+\infty}\max_{S^n} u(t,u_0)=+\infty.
  \end{equation}
  Now, for each $u_0\in L_{\gamma_0}$, we fix a sufficiently large $T>0$ and set $\beta=\beta(T)=[\max_{S^n} u(T,u_0)]^{-1}$. By \eqref{ugoestoinfty}, we immediately have
  \begin{equation}\label{betagoesto0}
    \lim_{T\rightarrow+\infty}\beta=0.
  \end{equation}
  By following the same proof of Proposition 7.1 on \cite[Page477-478]{cx} , $T$ can be chosen continuously depending on the initial data $u_0$. Hence, $\beta$ is continuously depending on $u_0$ either.
  
 \noindent Now, define
  \begin{equation*}
    u_s=
    \begin{cases}
      u(2sT,u_0),&0\leqslant s\leqslant\frac12,\\
      \bigg[\frac{(2-2s)\big(\beta u(T,u_0)\big)^{2^*}+(2s-1)}{(2-2s)\beta^{2^*}+2s-1}\bigg]^\frac{1}{2^*},&\frac12<s\leqslant1.
    \end{cases}
  \end{equation*}
  Then, there holds the claim

  \noindent{\bf Claim}: The function $u_s$ satisfies $1^\circ$. $\dashint_{S^n}u_s^{2^*}~d\mu_{S^n}=1$, $2^\circ$. $\dashint_{S^n}fu_s^{2^*}~d\mu_{S^n}>0$ and $3^\circ$. $E_f[u_s]\leqslant \gamma_0$.
  \vspace{0.7em}

  \noindent{\itshape Proof of Claim}: From the volume-preserving property of the flow \eqref{eeforu}, Lemma \ref{pospro} and decay property of the energy functional $E_f[u]$, it follows that $u_s$ fulfills the said properties in the claim for $0\leqslant s\leqslant\frac12$. Therefore, we are left to show the claim for $\frac12<s\leqslant 1$.

  \noindent$1^\circ$. By a direct computation and volume-preserving property of the flow \eqref{eeforu}, we conclude that
  $$\dashint_{S^n}u_s^{2^*}~d\mu_{S^n}=\frac{(2-2s)\beta^{2^*}\dashint_{S^n}u(T,u_0)^{2^*}~d\mu_{S^n}+2s-1}{(2-2s)\beta^{2^*}+2s-1}=1.$$

  \noindent$2^\circ$. It follows from a direct computation, Lemma \ref{pospro} and assumption (i) in Theorem \ref{main} that
  $$\dashint_{S^n}fu_s^{2^*}~d\mu_{S^n}=\frac{(2-2s)\beta^{2^*}\dashint_{S^n}fu(T,u_0)^{2^*}~d\mu_{S^n}
  +(2s-1)\dashint_{S^n}f~d\mu_{S^n}}{(2-2s)\beta^{2^*}+2s-1}>0.$$

  \noindent$3^\circ$. We set
  $$w_s=\bigg[(2-2s)\big(\beta u(T,u_0)\big)^{2^*}+(2s-1)\bigg]^\frac{1}{2^*}.$$
   Since the energy functional $E_f[u]$ is scale-invariant, we have
  \begin{equation}
    \label{scaleinvariant}
    E_f[u_s]=E_f[w_s]=\frac{\dashint_{S^n}c_n|\nabla w_s|_{S^n}^2+n(n-1)w_s^2~d\mu_{S^n}}{\bigg(\dashint_{S^n}fw_s^{2^*}~d\mu_{S^n}\bigg)^\frac{n-2}{n}}:=\frac{I}{(II)^\frac{n-2}{n}}.
  \end{equation}
  {\bf Estimate of $I$}: Notice that a simple calculation shows that
  \begin{eqnarray*}
    \nabla w_s&=&\beta(2-2s)\big(\beta u(T,u_0)\big)^\frac{n+2}{n-2}\bigg[(2-2s)\big(\beta u(T,u_0)\big)^{2^*}+(2s-1)\bigg]^{-\frac{n+2}{2n}}\nabla u(T,u_0)\\
    &=&\beta(2-2s)^\frac{n-2}{2n}\bigg[\frac{(2-2s)\big(\beta u(T,u_0)\big)^{2^*}}{(2-2s)\big(\beta u(T,u_0)\big)^{2^*}+(2s-1)}\bigg]^\frac{n+2}{2n}\nabla u(T,u_0),
  \end{eqnarray*}
  which implies that
  \begin{eqnarray*}
  |\nabla w_s|_{S^n}^2&=&\beta^2(2-2s)^\frac{n-2}{n}\bigg[\frac{(2-2s)\big(\beta u(T,u_0)\big)^{2^*}}{(2-2s)\big(\beta u(T,u_0)\big)^{2^*}+(2s-1)}\bigg]^\frac{n+2}{n}|\nabla u(T,u_0)|_{S^n}^2\\
  &\leqslant&\beta^2(2-2s)^\frac{n-2}{n}|\nabla u(T,u_0)|_{S^n}^2.
  \end{eqnarray*}
  Moreover, by an elementary inequality, we get that
  $$w_s^2=\bigg[(2-2s)\big(\beta u(T,u_0)\big)^{2^*}+(2s-1)\bigg]^\frac{n-2}{n}\leqslant\beta^2(2-2s)^\frac{n-2}{n}u^2(T,u_0)+(2s-1)^\frac{n-2}{n}.$$
  Combining the two estimates above yields
  \begin{equation}
    \label{estimateofI}
    I\leqslant\beta^2(2-2s)^\frac{n-2}{n}E[u(T,u_0)]+n(n-1)(2s-1)^\frac{n-2}{n}.
  \end{equation}
  {\bf Estimate of $II$}: It is easy to see that
  \begin{equation}
    \label{estimateofII}
    II=\dashint_{S^n}fw_s^{2^*}~d\mu_{S^n}=\beta^{2^*}(2-2s)\dashint_{S^n}fu^{2^*}(T,u_0)~d\mu_{S^n}+(2s-1)\dashint_{S^n}f~d\mu_{S^n}.
  \end{equation}
  Plugging \eqref{estimateofI} and \eqref{estimateofII} into \eqref{scaleinvariant} gives
  \begin{equation}\label{E_f[u_s]bound}
  E_f[u_s]=E_f[w_s]\leqslant\frac{\beta^2(2-2s)^\frac{n-2}{n}E[u(T,u_0)]+n(n-1)(2s-1)^\frac{n-2}{n}}{\bigg(\beta^{2^*}(2-2s)\dashint_{S^n}fu^{2^*}(T,u_0)~d\mu_{S^n}
  +(2s-1)\dashint_{S^n}f~d\mu_{S^n}\bigg)^\frac{n-2}{n}}.
  \end{equation}
  Notice that from the volume-preserving property of flow \eqref{eeforu} and \eqref{poslowerbound}, it follows that
  $$\alpha_1:=\bigg(\frac{n(n-1)}{E_f[u_0]}\bigg)^\frac{n}{n-2}\leqslant\dashint_{S^n}fu^{2^*}(T,u_0)~d\mu_{S^n}\leqslant\max_{S^n}|f|:=\alpha_2.$$
  Moreover, by the fact that $E_f[u(T,u_0)]\leqslant\gamma_0$, we get
  $$E[u(T,u_0)]\leqslant\gamma_0\bigg(\dashint_{S^n}fu^{2^*}(T,u_0)~d\mu_{S^n}\bigg)^\frac{n-2}{n}\leqslant\gamma_0\alpha_2^\frac{n-2}{n}.$$
  Substituting all the estimates above into \eqref{E_f[u_s]bound} yields
  $$E_f[u_s]\leqslant\frac{\beta^2(2-2s)^\frac{n-2}{n}\gamma_0\alpha_2^\frac{n-2}{n}+n(n-1)(2s-1)^\frac{n-2}{n}}
  {\bigg(\beta^{2^*}(2-2s)\alpha_1+(2s-1)\dashint_{S^n}f~d\mu_{S^n}\bigg)^\frac{n-2}{n}}.$$
  By letting $T\rightarrow+\infty$ in the estimate above, observing the fact \eqref{betagoesto0} and the choice of $\gamma_0$, we have
  $$\lim_{T\rightarrow+\infty}E_f[u_s]\leqslant\frac{n(n-1)}{\bigg(\dashint_{S^n}f~d\mu_{S^n}\bigg)^\frac{n-2}{n}}<\gamma_0.$$
  By choosing $T>0$ large enough, we thus complete the proof of claim.

  From the claim, it follows that $u_s\in L_{\gamma_0}$ for all $0\leqslant s\leqslant1$. Moreover, by the definition of $u_s$, it is easy to see that $u_s=u_0\in L_{\gamma_0}$ for $s=0$ and $u_s\equiv1$ for $s=1$. Hence, $u_s$ induces a contraction within $L_{\gamma_0}$. Therefore, we complete the proof of (i).

  Since the proof of (ii), (iii) and (iv) follows exactly the same proof as in \cite[Proposition 7.1]{cx}, we omit the details.
\end{proof}
To end this subsection, we will apply Proposition \ref{homotopy} to prove Theorem \ref{main}.

\noindent{\bf\itshape Proof of Theorem \ref{main}}: Suppose the contrary, namely, $f$ cannot be realized as the scalar curvature of any conformal metric $g$ on $S^n$. A suitable choice of $\gamma_0$ in part (i) of Proposition \ref{homotopy} shows that $L_{\gamma_0}$ is contractible. In addition, the flow \eqref{eeforu} defines a homotopy equivalence of the set $\mathscr{E}_0=L_{\gamma_0}$ with a set $\mathscr{E}_\infty$ whose homotopy type is that of a point with $n-\mbox{ind}_f(x)$ dimensional cells attached for every critical point $x$ of $f$ on $S^n$ where $f(x)>0$ and $\Delta_{S^n}f(x)<0$. It then follows from \cite[Theorem 4.3]{ch} that
\begin{equation}
  \label{systemequation}
  \sum_{i=0}^ns^im_i=1+(1+s)\sum_{i=0}^ns^ik_i
\end{equation}
holds for the Morse polynomials of $\mathscr{E}_0$ and $\mathscr{E}_\infty$, where $k_i\geqslant0$ and $m_i$ are given in \eqref{Morseindex1}. By equating the coefficients in the polynomials on the left and right hand side of \eqref{systemequation}, we obtain a set of non-trivial solutions of \eqref{Morseindex}, which violates the hypothesis in Theorem. We thus obtain the desired contradiction and the proof of Theorem \ref{main} is completed. Furthermore, by setting $s=-1$ in \eqref{systemequation} we can obtain \eqref{indexcounting} and thus the assertion in Corollary \ref{corollary} holds.\qed

\subsection{Proof of Theorem \ref{main1}}
 Likewise, we adopt the contradiction argument, that is, for any time sequence $(t_k)_k$, case (ii) occurs to the corresponding sequential metrics $(g_k)_k$. By \cite[Lemma 2.2]{lz}, we see that $u$ is a $G$-invariant function if the initial data $u_0\in X_f$ is a $G$-invariant function. Fix any $G$-invariant initial data $u_0\in X_f$, by the uniqueness of the solution of the flow \eqref{eeforu} and the decay of $E_f[u]$, we can assume that
\begin{equation}\label{energybound}
E_f[u(t)]<E_f[u_0],\quad\mbox{for}~~\forall t\in(0,+\infty),
\end{equation}
 Since we have assumed that case (ii) in Theorem \ref{cc} occurs to the corresponding sequential metrics $(g(t_k))_k$, the blow-up behavior of $(g(t_k))_k$ in Proposition \ref{blowupbehavior} will happen. In particular, it follows from Proposition \ref{blowupbehavior} that
\begin{equation}\label{integralonB_r(Q)}
\lim_{k\rightarrow+\infty}\omega_n^{-1}\int_{B_r(Q)}fu_k^{2^*}~d\mu_{S^n}=f(Q)=\frac{n(n-1)}{\lambda_\infty},
\end{equation}
where $r>0$ is arbitrary and $Q$, depending on the choice of $u_0$, is the unique concentration point in Proposition \ref{blowupbehavior}, which also implies that for any $y\in S^n$, if $Q\notin B_r(y)$ for some $r>0$, then
\begin{equation}
  \label{integralonoutsideofB_r(Q)}
  \lim_{k\rightarrow+\infty}\int_{B_r(y)}fu_k^{2^*}~d\mu_{S^n}=0.
\end{equation}
Now, we split our argument into two cases

\noindent{\bf Case 1}. $\Sigma=\emptyset$. If this case happens, then we can find $\theta\in G$ such that $\theta(Q)\neq Q$. Since $f$ and $u_k$ are $G$-invariant, we conclude, by \eqref{integralonB_r(Q)} and change of variables, that
\begin{eqnarray*}
  \lim_{k\rightarrow+\infty}\omega_n^{-1}\int_{B_r(\theta(Q))}fu_k^{2^*}~d\mu_{S^n}
  &=&\lim_{k\rightarrow+\infty}\omega_n^{-1}\int_{B_r(Q)}(f\circ\theta(y))(u_k\circ\theta(y))^{2^*}~d\mu_{S^n}\\
  &=&\lim_{k\rightarrow+\infty}\omega_n^{-1}\int_{B_r(Q)}fu_k^{2^*}~d\mu_{S^n}\\
  &=&\frac{n(n-1)}{\lambda_\infty}\neq0
\end{eqnarray*}
On the other hand, as $\theta(Q)\neq Q$, we can find $r>0$ small enough such that $Q\notin B_r(\theta(Q))$. Then, by \eqref{integralonoutsideofB_r(Q)}, we have
$$\lim_{k\rightarrow+\infty}\omega_n^{-1}\int_{B_r(\theta(Q))}fu_k^{2^*}~d\mu_{S^n}=0,$$
which is a contradiction.

\noindent{\bf Case 2}. $\Sigma\neq\emptyset$. {\itshape When the condition $1^\circ$ holds}: By Remark \ref{rk2}, we can choose the initial data $u_0\equiv1$ which is obviously a $G$-invariant function. If $Q\notin \Sigma$, then we can repeat the argument in Case 1 to obtain a contradiction which shows that $u_k$ is bounded in $W^{2,p}(S^n,g_{S^n})$ for $p>2/n$. However, this, in turn, contradicts with our contrary assumption.  Hence, we must have $Q\in \Sigma$. To proceed, we need a refined estimate upon the number $\lambda_\infty$.
By the decay of $E_f[u]$, we have
$$\frac{E[u_k]}{(\dashint_{S^n}fu_k^{2^*}~d\mu_{S^n})^\frac{2}{2^*}}\leqslant E_f[u_1],$$
for all $k\geqslant1$, which implies, by sharp Sobolev inequality, that
\begin{equation*}
\dashint_{S^n}fu_k^{2^*}~d\mu_{S^n}\geqslant \bigg(\frac{n(n-1)}{E_f[u_1]}\bigg)^\frac{2^*}{2}.
\end{equation*}
From this, it follows that
\begin{eqnarray*}
  \lambda_k&=&E_f[u_k]\bigg(\dashint_{S^n}fu_k^{2^*}~d\mu_{S^n}\bigg)^{-\frac 2n}\\
  &\leqslant& E_f[u_1]\bigg(\frac{n(n-1)}{E_f[u_1]}\bigg)^{-\frac{2}{n-2}}\\
  &=&\bigg(\frac{E_f[u_1]}{(n(n-1))^{2/n}}\bigg)^\frac{n}{n-2}.
\end{eqnarray*}
By letting $k\rightarrow+\infty$ in the inequality above and \eqref{energybound}, we have
$$\lambda_\infty\leqslant\bigg(\frac{E_f[u_1]}{(n(n-1))^{2/n}}\bigg)^\frac{n}{n-2}
<\bigg(\frac{E_f[u_0]}{(n(n-1))^{2/n}}\bigg)^\frac{n}{n-2}=\frac{n(n-1)}{\dashint_{S^n}f~d\mu_{S^n}},$$
where we have used the fact $u_0\equiv1$ in the last equality. On the other hand, we know that $\lambda_\infty f(Q)=n(n-1)$. Hence, we conclude that
$$f(Q)>\dashint_{S^n}f~d\mu_{S^n},$$
which contradicts with our assumption $1^\circ$ in Theorem \ref{main1}.

\noindent{\itshape When the condition $2^\circ$ holds}: If $\max_\Sigma f\leqslant\dashint_{S^n}f~d\mu_{S^n}$, then we can repeat the argument as in the case of the condition $1^\circ$ holding. While $\max_\Sigma f>\dashint_{S^n}f~d\mu_{S^n}$, from \cite[Lemma 3.3]{lz}, we can choose a $G$-invariant function $u_0$ such that for sufficiently small $\epsilon>0$ there holds
\begin{equation}\label{chioceofu_0}
E_f[u_0]<\frac{n(n-1)}{\big(\max_\Sigma f\big)^\frac{n-2}{n}}+\epsilon,
\end{equation}
which implies, by the choice of $\gamma$, that
$$E_f[u_0]<\frac{n(n-1)}{\bigg(\dashint_{S^n} f~d\mu_{S^n}\bigg)^\frac{n-2}{n}}+\epsilon\leqslant\gamma.$$
This shows that $u_0\in X_f$. Once we have this fact, the conclusions in Proposition \ref{asymptoticbehavior} will hold. In particular, at the corresponding concentration point $Q$, there holds $\Delta_{S^n}f(Q)\leqslant0$. By following the argument in \cite[\S3c]{lz}(See (3.13) on P1615), we get that there exists a constant $\alpha>0$ such that
$$f(x_*)=\max_\Sigma f>f(Q)+\alpha.$$
Substituting this inequality into \eqref{chioceofu_0} yields
$$E_f[u_0]<\frac{n(n-1)}{\big(f(Q)+\alpha\big)^\frac{n-2}{n}}+\epsilon
=\frac{n(n-1)}{\big(f(Q)\big)^\frac{n-2}{n}}\bigg(\frac{f(Q)}{(f(Q)+\alpha)}\bigg)^\frac{n-2}{n}+\epsilon.$$
Since $\epsilon$ is sufficiently small, we have
$$\epsilon\leqslant\frac{n(n-1)}{\big(f(Q)\big)^\frac{n-2}{n}}\bigg[1-\bigg(\frac{f(Q)}{(f(Q)+\alpha)}\bigg)^\frac{n-2}{n}\bigg],$$
which implies, by Proposition \ref{asymptoticbehavior}, that
$$E_f[u_0]<\frac{n(n-1)}{\big(f(Q)\big)^\frac{n-2}{n}}=\lim_{k\rightarrow+\infty}E_f[u_k].$$
But this contradicts with the decay property of the energy functional $E_f[u]$. \qed

\section*{Acknowledgement}
  \noindent This project is supported by the ``Fundamental Research Funds for the Central Universities"

\end{document}